\documentclass[12pt]{amsart}
\usepackage[cp1251]{inputenc}
\usepackage{amsmath,amssymb}
\usepackage[shortlabels]{enumitem}
\usepackage{graphicx}
\usepackage{subcaption}

\newtheorem{thm}{Theorem}
\newtheorem{prop}[thm]{Proposition}
\newtheorem{cor}[thm]{Corollary}

\newtheorem{rem}[thm]{Remark}
\newtheorem{lem}[thm]{Lemma}

\def\A{\mathcal A}
\def\B{\mathcal B}
\def\C{\mathcal C}
\def\D{\mathcal D}

\def\R{\mathcal R}

\def\P{\mathcal P}
\def\S{\mathcal S}
\def\DD{\Delta}

\def\s{\smallskip}
\def\n{\noindent}
\def\v{\varphi}
\def\w{\wedge}
\def\t{\triangle}
\def\o{\overline}

\def\ss{\subset}

\title{Moving rectangular sofas in planar and spatial corridors}

\author{Oleg Mushkarov}
\address{O. Mushkarov\\
Institute of Mathematics and Informatics\\
Bulgarian Academy of Sciences\\
Acad. G. Bonchev 8, 1113 Sofia, Bulgaria}
\email{muskarov@math.bas.bg}

\author{Nikolai Nikolov}
\address{N. Nikolov\\
Institute of Mathematics and Informatics\\
Bulgarian Academy of Sciences\\
Acad. G. Bonchev 8, 1113 Sofia, Bulgaria
\vspace{1mm}
\newline Faculty of Information Sciences\\
State University of Library Studies and Information Technologies\\
Shipchenski prohod 69A, 1574 Sofia,
Bulgaria}
\email{nik@math.bas.bg}

\thanks{The second named author was partially supported by the Bulgarian National Science
Fund, Ministry of Education and Science of Bulgaria under contract KP-06-N82/6.}

\begin{document}



\begin{abstract}
We consider eight natural planar corridors, including the standard
$ \mathrm{L}$-shaped  one, and characterize the rectangles that
can move around their corners. As a bi-product we describe
completely the corresponding rectangles with maximum area, as well
as the rectangular parallelepipeds with maximum volume that can
move around the corners of the spatial analogues of the considered
eight planar corridors.
\end{abstract}

\maketitle

\section{Introduction}
A well-known version of the famous \emph{moving sofa problem }
\cite{LM, G,R, B} is to characterize the rectangles that can move
around the right-angled corner in an $ \mathrm{L}$-shaped corridor
with given widths. This problem has been particularly solved in
\cite{Mo} (cf. also \cite{Ka} and \cite{Mi}), where the longest
such rectangle of a given width has been found in terms of the
smallest positive root of a six degree algebraic equation.

The goal of this paper is to consider eight planar corridors,
including the standard $ \mathrm{L}$-shaped one, and to completely
solve the aforementioned problem for each of them. More precisely,
let
$$\D=\{x=a, 0<y<b\} \ , \ \A_0=\D\cup\{0<x<a\} \ , \
\B_0=\D\cup\{x>a\} ,$$
$$\A_1=A_0\cap\{y>0\} \ ,  \ \B_1=B_0\cap\{y>0\} \ , \ \B_2=\B_0\cap\{y<b\},
\ \B_3=\B_0\cap\{0<y<b\}.$$ Given a pair $(i,j)\in \{0,1\}
\times \{0,1,2,3\}$ we denote by $\R_{ij}$ the set of open
rectangles that can move around the corner in the corridor
$\C_{ij}=\A_i\cup\B_j.$\footnote{For a rectangle $R$ in the plane
this means that there is a continuous path $\Phi_t$ in the
Euclidean  group $E(2)$ of rigid motions such that
$\Phi_0(R)\ss\A,$ $\Phi_1(R)\ss\B$ and $\Phi_t(R)\ss\C_{ij}$ for
any $t\in[0,1].$} (Fig. 1)

\s

 \begin{figure}[!t]
 \begin{center}
 \includegraphics[width=0.5\textwidth]{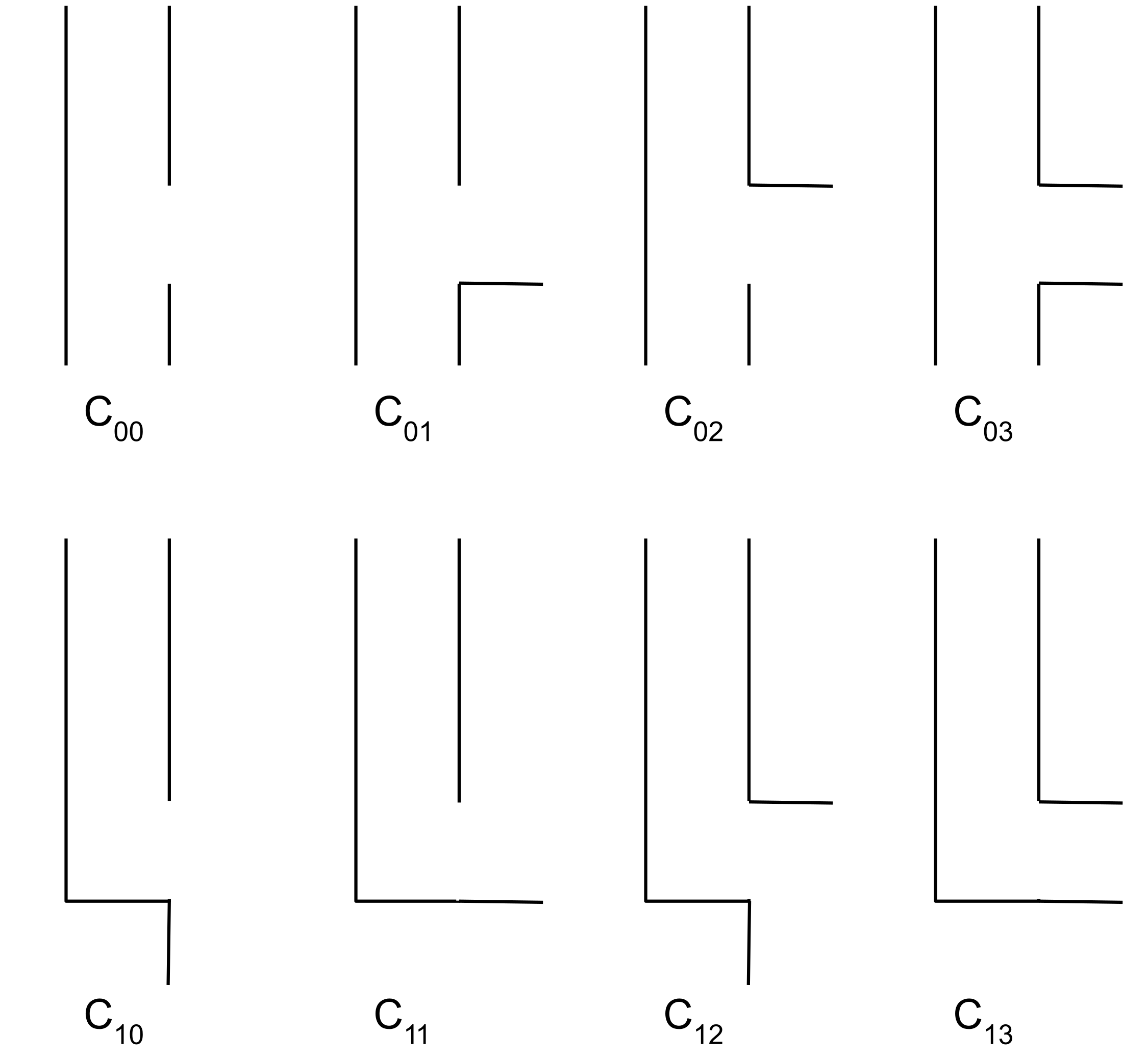} \caption{The
 corridors $\C_{ij}$.} \label{fig1}
 \end{center}
 \end{figure}

\s

For arbitrary $a,b,c>0$ let $m(a,b,c)=\min_\DD
f_{(a,b,c)},$ where $\DD =[0,\pi/2]$ and
\begin{equation}\label{1}
f_{(a,b,c)}(t)=a\sin t+b\cos t-c\sin t\cos t.
\end{equation}
Set also $l=\sqrt{a^2+b^2}, h=ab/\sqrt{a^2+b^2}$ and $ a\vee
b=\max\{a,b\}, a\w b=\min\{a,b\}.$

Our main result is the following theorem which characterizes
completely the rectangles from the sets $\R_{ij}$ and shows that
these sets are divided into four nonintersecting groups.

\begin{thm} \label{t1} The following identities hold true: $\R_{00}=\R_{01}=\R_{02}$,
$\R_{10}=\R_{12}$, $\R_{11}=\R_{13}$, $\R_{00}=\R_{10}\cup\R_{03}$,
$\R_{13}=\R_{10}\cap\R_{03}$. Moreover, a rectangle with side
lengths $c$ and $d , c\ge d , $ belongs to the set:

(i) $\R_{0i}, i=0,1,2,$ if and only if $d\le a\w b, cd\le
a$;

(ii) $\R_{03}$, if and only if

($ii_1$) $h \le d\le a\w b, cd\le ab$, or ($ii_2$) $d\le h\w
m(a,b,c);$

(iii) $\R_{1i}, i=0,2,$ if and only if

($iii_1$) $d\le h, cd\le ab,$ ($iii_2$) $h\le d\le a\w b,
c\le a\vee b,$ or $(iii_3$) $h\le d\le m(a,b,c);$

(iv) $\R_{1i}, i=1,3$ if and only if

($iv_1$) $c\le a\vee b, d\le a\w b,$ or ($iv_2$) $d\le
m(a,b,c).$
\end{thm}

Note that the above conditions do not depend on the position of
the rectangle $R$ in the corridor $A_i$. Observe also  that the
case $d=0$ in Theorem \ref{t1} corresponds to the so-called
\emph{ladder problem} for $\mathrm{L}$-shaped corridors \cite{Ka}.
Its solution for the corridors $\C_{ij}$ is given by

\s

\begin{cor} \label{cr1} The length of the longest line segment that can move
around the corner in the corridor $\C_{03} , \C_{11} ,$or
$\C_{13}$ is $(a^{2/3}+b^{2/3})^{3/2},$ whereas for the remaining
five corridors, this can be done for a segment of arbitrary
length.
\end{cor}

Another consequence of Theorem \ref{t1} (using inequality
(\ref{8})) is the following

\s

\begin{cor} \label{cr2}  The rectangles from the sets  $\R_{ij}$ have maximum area $ab$.
Moreover, the maximum area rectangles from the set:

(i) $\R_{0i} , i=0,1,2,$ have side lengths $ab/d$ and $d,$ where
$0<d\le a\w b$;

(ii) $\R_{03}$ have side lengths $ab/d$ and $d,$ where $h\le d\le
a\w b $;

(iii) $\R_{1i} , i=0,2,$ have side lengths $a$ and $b$, or $ab/d$
and $d,$ where  $0<d\le h$;

(iv) $\R_{1i} , i=1,3, $ have side lengths $a$ and $b$, or $l$ and
$h$.
\end{cor}

In Section 4 of the paper we consider the spatial corridors
$\S_{ij}=\C_{ij}\times (0,c) , c>0,$ and describe completely the
rectangular parallelepipeds  with maximum volume that can move
around the corner in $\S_{ij}$. Denoting the set of these
parallelepipeds by $\P_{ij}$ we prove the  following theorem which
gives a partial answer to \cite[Question 4, p. 200]{Mo}.

\begin{thm} \label{t2} Any $P\in
\P_{ij}$ has volume $abc$ and $P=R\times (0,c)$ for a rectangle
$R\in \R_{ij}$ with maximum volume.
\end{thm}

As a consequence of Theorem \ref{t2} and Corollary \ref{cr1} we
obtain the following solution of the \emph{spatial leader
problem}.

\begin{cor} \label{cr3} The length of the longest line segment that can move
around the corner in the corridor $\S_{03} , \S_{11} ,$ or
$\S_{13}$ is $((a^{2/3}+b^{2/3})^{3}+c^2)^{1/2},$ whereas for
the remaining five corridors, this can be done for a segment of
arbitrary length.
\end{cor}

The paper is organized as follows. In Section 2 we characterize
the rectangles from the sets $\R_{00}$ and $\R_{13}$ (Proposition
\ref{p1} and Proposition \ref{p2}, respectively) and in Section 3
we use these results to prove Theorem \ref{t1}. Let us mention two
difficulties that arise here. The first is the justification of
the corresponding necessary conditions, since when moving a
rectangle in a corridor it can rotate around the corner in either
direction. Note that in \cite {Mo} and \cite{Ka} the case of
``anti-rotation'' (see the proof of Proposition \ref{p2}) has not
been considered and a rigorous proof of the necessary condition
for $\mathrm{L}$-shaped corridors is missing. The second
difficulty is that for proving the sufficiency of the conditions
for the side lengths of a rectangle in $R_{03}$ one has to combine
the movements used for rectangles in $\R_{00}$ and $\R_{13}$.

The proof of Theorem \ref{t2}, given in Section 4, is two-step. We
first prove the inequality $V_P\le abc$, for a rectangular
parallelepiped $P$ that can move around the corner of a corridor
$S_{ij}$. This is done by analyzing the position of $P$ at the
last moment until at least six of its vertices lie in
$\A_{i}\times (0,c)$. Then we prove that if $V_P=abc$, then
$P=R\times (0,c)$ for a rectangle $R\in \R_{ij}$ with maximum
volume. This follows by combining the above analysis and the
following \s

\begin{lem} \label{l1} Let $\alpha$ and $\beta$ be parallel planes
in the space and $P$ be a rectangular parallelepiped  with two
opposite faces on these planes. Then at any moment of a movement
of $P$ in the layer between $\alpha$ and $\beta$ these faces lie
on them.
\end{lem}

The Appendix at the end of the paper collects some technical facts
used in the proofs of the results stated in Section 1.

\section{Moving rectangles in the corridors $\C_{00}$ and $ \C_{13}$ }

To prove Theorem \ref{t1} we first characterize the rectangles
from the sets $\R_{00}$  and  $ \R_{13}$.

\begin{prop}\label{p1} A rectangle with side lengths $c$ and $d,$ $c\ge d,$
belongs to the set $\R_{00}$ if and only if $d\le a\w b $ and
$cd\le ab.$
\end{prop}

\n{\it Proof.} We first prove the \emph{only if part} of the
proposition. Let $A=(a,0)$ and $C=(a,b)$ be the ends of the
segment $\mathcal{D}$ and $R=KLMN$ be a rectangle with $|KL|=d$
and $|KN|=c$. Note that if $d>a$, then $R \nsubseteq \A_0$ and if
$d>b$, then it is not possible to move $R$ into $\B_0$ trough
$\D$. Hence $d\le a\w b $.

Suppose now that it is possible to move $R$ from $\A_0$ to $\B_0$
trough $\D$. Consider the first moment when at least two vertices
of $R$ lie in $\B_0$. Suppose first that at this moment there is
no vertex of $R$ lying in $\{x>a\}$. Then its side is lying on
the segment $\mathcal{D}$ and it is clear that $cd\le ab.$ The
other possibility is that there is only one vertex of $R$ in
$\{x>a\}$, say $K$. Then either $L$ or $M$, say $L$, lies on
$\mathcal{D}$. We can assume that the line $KN$ meets $\D$ and the
line $\{x=0\}$ at some points $K_0$ and $N_0$, respectively, and
let the line $LM$ meets $\{x=0\}$ at $M_0$ (Fig.~\ref{fig2A}).

\s

 \begin{figure}[!t]
 \begin{center}
      \begin{subfigure}[]{0.38\textwidth}
      \includegraphics[width=\textwidth]{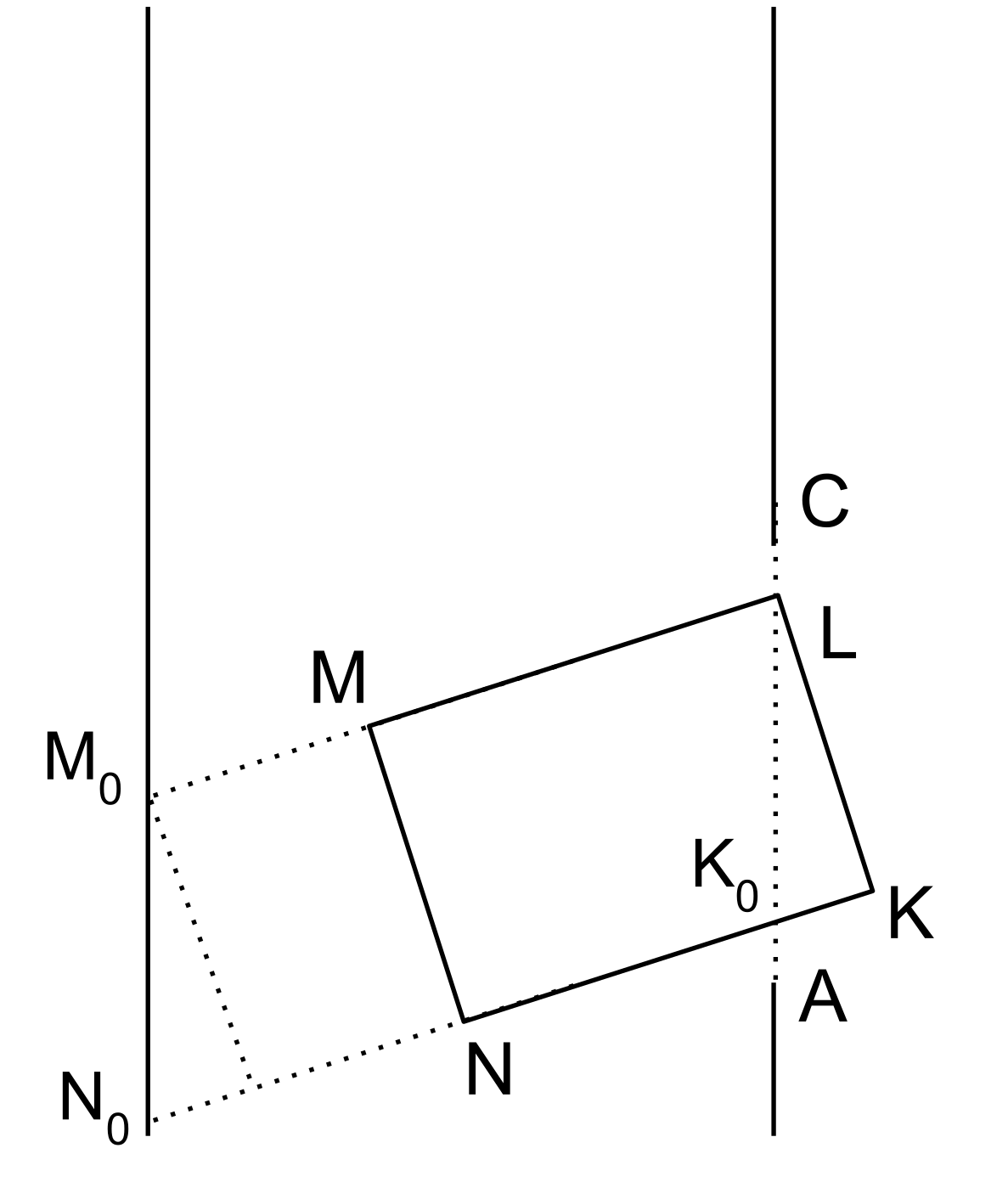}
      \caption{}
     \label{fig2A}
     \end{subfigure}
     \hspace{1cm}
     \begin{subfigure}[]{0.38\textwidth}
     \includegraphics[width=\textwidth]{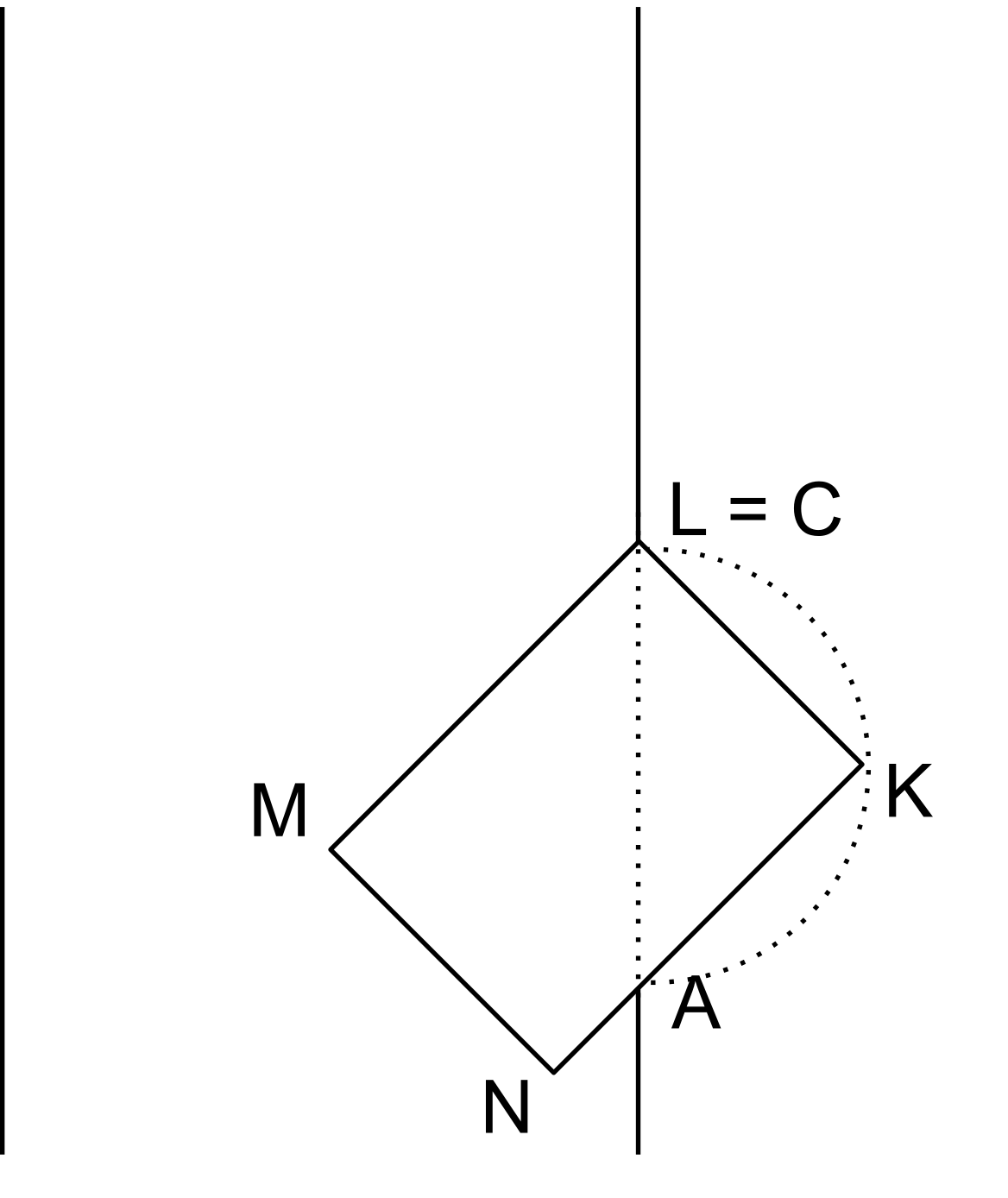}
     \caption{}
     \label{fig2B}
     \end{subfigure}
 \caption{The corridor $\C_{00}$.}
 \label{fig2}
 \end{center}
 \end{figure}

\s

Then
$$ab\ge a|K_0L|=S_{K_0LM_0N_0}\ge S_{KLMN}=cd $$
and the \emph{only if} part is proved.

\s

Conversely, let $c\ge d, d\le a\w b $ and $cd\le ab.$ Then it
is easy to see that for any position of $R$ in $\A_0$ one can move
it in a way that one of its longest sides, say $KN$, lies on the
line $\{x=a\}, K=C$ and $N$ is below the point $C$. If $c\le b$,
then one can move $R$ to $\B_0$ by a translation parallel to the
line $\{y=0\}$. Let now $c> b$. Then one moves $R$ such that the
vertex $K$ is moving on the semicircle with diameter $AC$, lying
in the half-plane $\{x>a\}$. Consider the moment when $|KC|=d ,$
i.e. $L=C$ (Fig.~\ref{fig2B}). We will show that then $M\in \A_0$
and hence one can move $R$ in $\B_0$ by a translation parallel to
its side $ML$. Indeed, it is easy to check that if $R$ is in the
above position, then the distance from $M$ to the line $\{x=a\}$
is equal to $cd/b$ and $M\in \A_0$ since $cd/b\le a$.\qed

\s

In the next proposition we characterize the rectangles that can
move around the corner of an $\mathrm{L}$-shaped corridor. This is
done in terms of the function $f_{(a,b,c)}$,  defined in
(\ref{1}).  Its geometric meaning is the following -- the value of
$f_{(a,b,c)}$ at $t\in \DD$ is equal to the distance from the
point $C(a,b)$ to the line with slope $-\tan t$ that cuts a
segment of length $c$ from the first quadrant. It is clear that
$f_{(a,b,c)}$ is a continuous function which is increasing in $a$
and $b$, and decreasing in $c$.

\begin{prop}\label{p2}  A rectangle with side lengths $c$ and $d ,
c\ge d , $ belongs to the set $\R_{13}$ if and only if $c\le
a\vee b$ and $ \ d\le a\w b,$ or $d\le m(a,b,c) .$
\end{prop}

\n{\it Proof.} Note as above that a rectangle $R$ with side
lengths $c$ and $d , c\ge d , $ can lie in  $\A_1$ (resp. $\B_3$)
precisely when $d\le a$ (resp. $d\le b.$) Moreover, it can be
moved in $\A_1$ (resp. $\B_3$) to a position where two of its
sides lie on the coordinate axes. Such positions of $R$ are called
{\it good}. Hence we can assume that any movement of $R$ from
$\A_1$ into $\B_3$ around the corner $C$ starts and ends at good
positions and if $c\ge a\vee b$ a  side of $R$ with length $c$
lies on $Oy^+$ at its initial position and on $Ox^+$ at the final
one. (Note that that $R$ can be moved in the corridor $\A_1$ so
that it lies on $Oy^+$ on sides of lengths $c$ and $d$ if and only
if $c^2+d^2\le a^2$.)

Let $\overrightarrow{e_1}=(1,0)$ and
$\overrightarrow{e_2}=(0,1)$ be the unit coordinate vectors and
let $KLMN $ be the initial position of $R$ with $M=(0,0)$ and
$\overrightarrow{ML}=c\overrightarrow{e_2}$. We will say that a
movement of $R$ around $C$ is a \emph{rotation} (resp.
\emph{anti-rotation}) if $\overrightarrow{ML} =
-c\overrightarrow{e_1}$ (resp. $\overrightarrow{ML} =
c\overrightarrow{e_1}$) at its final position.

\begin{lem}\label{l2} A rectangle $R\ss \A_1$ with side lengths
$c$ and $d$ (not necessarily $c\ge d$) can be moved into $\B_3$ by
a rotation around the corner $C$ if and only if $d\le m(a,b,c).$
\end{lem}

\begin{proof} We first prove the \emph{only if part} of the
lemma.
Let $R\ss \A_1$ be a rectangle with side lengths $c$
and $d$ that can move in $\B_3$ by a rotation around
the corner $C$. It is clear that $d\le a\w b $ and we have to
prove that if $c>a\vee b ,$ then $d\le m(a,b,c).$
Let $R=KLMN$ be a rectangle in a good position with vertices as
above. For any $t\in\DD^0$ denote by $A_t\in Ox^+$ and $B_t\in
Oy^+$ the points such that $C\in A_tB_t$ and $\angle OA_tB_t=t.$
Consider the moment when the sides of $R$ with length $c$ are
parallel to $A_tB_t$ and move  $R$ by translations parallel to
$Oy$ and  $Ox$ so that $L\in Oy^+$ and $M\in Ox^+$. Consider the
closer to $O$ tangent  to $k(C,d)$ which is parallel $A_tB_t$ and
denote by $A$ and $B$ be its intersection points with $Ox^+$ and
$Oy^+$, respectively. It is clear that $|LM|\leq |AB|$ since
otherwise it follows by convexity that $R\subsetneq \C_{13}$.
Hence $d\leq f_{(a,b,c)}(t)$ for any $t\in\DD^0$ and therefore
$d\le m(a,b,c).$

We now prove the \emph{if part} of the lemma.
Recall that $R$ is
in a good initial position. Also, the inequality  $d\le m(a,b,c)$
implies $d\le f_{(a,b,c)}(0)=b$ and $d\le f_{(a,b,c)}(\pi/2)=a$.
Thus if $c\le a\vee b ,$ then $R\ss\B_3$. So, we have to show that
if $c\ge a\vee b$ and $d\le m(a,b,c),$ then we can move $R$ to
$\B_3$ by a rotation around $C$.
For $t\in \DD$ denote by $R_t$ the rectangle with vertices
$L_t=(0,c\cos t), M_t=(c\sin t,0), N_t=(d\cos t+c\sin t, d\sin t),
K_t=(d\cos t, d\sin t+c\cos t).$ The inequality $d\le m(a,b,c)$
means that $d\le f_{(a,b,c)}(t)$ for any $t\in\DD$ and therefore
the distance from $C$ to the line $L_tM_t$ is $\ge d$. Hence $R$
moves from $\A_1$ into $\B_3$ by rotation around $C$.
\end{proof}

For any $c\ge a\vee b$ we set $t_1=\arccos(b/c), t_2=\arcsin(a/c)
, \ \widetilde{\DD}=[t_1,t_2] $ and $\widetilde{m}(a,b,d)=
\min_{\widetilde{\DD}}f(a,b,d) .$

\begin{lem}\label{l3} A rectangle $R\ss \A_1$ with side lengths
$c$ and $d,$ $c\ge a\vee b,$ can be moved into $\B_3$ by an
anti-rotation around the corner $C$ if and only if $c\le
\widetilde{m}(a,b,d).$
\end{lem}

\begin{proof}
We first prove the \emph{only if part} of the lemma. Let a
rectangle $R\ss \A_1$ with side lengths $c$ and $d,$ $c\ge a\vee
b,$ can be moved into $\B_3$ by an anti-rotation around the corner
$C$. Note that for any $t\in \DD$ there is a moment when the angle
between a side of $R$ with length $d$ and $Ox$ is $t$. In
particular, at some moment such a side of $R$ is perpendicular to
$OC$ and therefore $c^2\le a^2+b^2$ which implies that $t_1\le
t_2$. We note that for $t\in [0,t_1]$ the rectangle $R$ lies in
the corridor $\A_1$, whereas for $t\in [t_2,\pi/2]$ it lies in the
corridor $\B_3$. Hence it is enough to consider only the positions
of $R$ where $t\in \widetilde{\DD}$. For any such $t$ there is a
moment when the sides of $R$ with length $d$ are parallel to the
tangent to the circle $k(C;c)$ with slope $-\tan t$ and we
conclude as in the proof of the \emph{only if part} of Lemma
\ref{l2} that $c\le \widetilde{m}(a,b,d)$.

To prove the \emph{if part } of the lemma suppose that $c\le
\widetilde{m}(a,b,d)$ and note that this inequality can be written
as
\begin{equation}\label{2}
d\le g_{(a,b,c)}(t)=\frac{a\sin t+b\cos t-c}{\sin t\cos t} .
\end{equation}
It is easy to check that $g_{(a,b,d)}(t)$ is the length of the
segment which the tangent to the circle $k(C;d)$ and slope $-\tan
t$ cuts from the first quadrant. In particular, we have from
(\ref{2}) that
$$ d\le g_{(a,b,c)}(t_1)=\frac{c(a-\sqrt{c^2-b^2)}}{b}\leq a $$
and
$$ d\le g_{(a,b,c)}(t_2)=\frac{c(b-\sqrt{c^2-a^2)}}{a}\leq b $$
since $c\ge a\vee b .$

Let now $R=KLMN$ be a rectangle with side lengths $c$ and $d,$
$c\ge a\vee b,$ which is in a good position with $M=O(0,0)$ and
$\overrightarrow{ML} = c\overrightarrow{e_2} $. For any $t\in
\widetilde{\DD}$ consider the rectangle $R_t=K_tL_tM_tN_t$ with
vertices $K_t=(d\cos t+c\sin t, c\cos t) , L_t=(c\sin t, c\cos
t+d\sin t) , M_t=(0,d\sin t) , N_t=(d\cos t,0) $. It is easy to
check that the function $x(t)=d\cos t+c\sin t$ is increasing on
the segment $[0,t_1]$ since $d\le a$ and $x(t_1)\le a $ since
$c\le f_{(a,b,d)}(t_1)$. Hence $K_t\in \A_1$ for any $t\in
[0,t_1]$ which implies that $R_t \ss \A$ for $t\in [0,t_1]$.
Similarly $R_t \ss \B$ for $t\in [t_2, \pi/2)$. It follows also
that $R_t \ss \C$ for $t\in [t_1,t_2],$ since the inequality $c\le
f_{(a,b,d)}(t) $ means that the point $C$ is at a distance at
least $d$ from the line $M_tN_t$.
\end{proof}

\begin{rem}\label{r1} The above proof easily implies that Lemma \ref{l3}
remains also true if $c<b$ or/and $c<a$, replacing $t_1$ by $0$
or/and $t_2$ by $\pi/2$ and adding the condition $ \ d\le a\w b.$
\end{rem}

We are now ready to prove Proposition \ref{p2}. It follows from
Lemma 1 and Lemma 2 that it is enough to prove that if $c> a\vee b
$ and a rectangle in $\A_1$ with side lengths  $c$ and $d, c\ge d
,$ can be moved into $\B_3$ by anti-rotation around the corner
$C$, then the same can be done by a rotation around $C$. In other
words, if $c\le \widetilde{m}(a,b,d)$, then
 $d\le m(a,b,c).$ Writing  the first inequality in the form
$d\le \min_{\widetilde{\DD}}g_{(a,b,c)}(t) ,$ where $g_{(a,b,c)}$
is the function defined by (\ref{2}), we have to prove that
\begin{equation}\label{3}
\min_{\DD}f_{(a,b,c)}(t) \ge \min_{\widetilde{\DD}}g_{(a,b,c)}(t) .
\end{equation}
To do this consider the function
$$ h(t)= f_{(a,b,c)}(\pi/2-t)= a\cos t+b\sin t-c\sin t\cos t,\ t\in\DD.$$
It is clear that $\min_{\DD}f_{(a,b,c)}(t)=\min_{\DD}h(t).$ In
addition, it follows easily from $c\ge a\vee b $ that $h(t)\ge
g(t)$ for any $t\in\DD$. On the other hand the function
$f_{(a,b,c)}(t)$, and hence $h(t)$, has a unique minimum in the
interval $\DD$ (see $\textbf{(A)}$ in Appendix) and let that of
$h(t)$ is attained at $t_{\ast}\in \DD$. Direct computations show
that

$$h'(t_1) = \frac{\sqrt{c^2-b^2}(\sqrt{c^2-b^2}- a)}{c} \ ,
\ h'(t_2) = \frac{\sqrt{c^2-a^2}(b-\sqrt{c^2-a^2})}{c} .$$ Since
$c^2\le a^2+b^2$ (see the proof of Lemma 2) we conclude that
$h'(t_1)\le 0 \le h'(t_2)$, i.e. $t_{\ast}\in \widetilde{\DD}$.
Then
$$\min_{\DD}f_{(a,b,c)}(t)= \min_{\DD}h(t)=h(t_{\ast})\ge g(t_{\ast})\ge
\min_{\widetilde{\DD}}g(t)$$
and inequality (\ref{3}) is proved. \qed

\begin{cor}\label{cr 4} Let $a=b=1$ and $R\ss\A_1$ be a rectangle with side
lengths $c$ and  $d.$ Then $R$ can be moved into $\B_3$ around the
corner $C$ by:

(i) a rotation if and only if $c\le 2(\sqrt2-1)$ and $d\le 1,$ or
$2(\sqrt2-1)\le c\le 2\sqrt2$ and $d\le\sqrt{2}-c/2;$

(ii) an anti-rotation if and only if $c\le\sqrt2-1/2$ and $d\le 1,$ or
$\sqrt2-1/2\le c\le\sqrt6-\sqrt2$ and $d\le 2\sqrt2-2c,$ or
$\sqrt6-\sqrt2\le c\le\sqrt2$ and $d\le c(1-\sqrt{c^2-1}).$
\end{cor}

\n{\it Proof.} (i) The derivative of the function $f_{(1,1,c)}$ is given by
$$f'_{(1,1,c)}(t)=(\cos t-\sin t)(1-c(\sin t+\cos t)).$$
Note that the equation
$$1-c(\sin t+\cos t)=0$$
has two distinct roots in $\DD$ (with sum $\pi/2$) for
$c\in(\sqrt2/2,1],$ and it has a double root $\pi/4$ for
$c=\sqrt2/2;$ otherwise, it has no roots in $\DD.$ Hence
$m(1,1,c)=f_{(1,1,c)}(0)=f_{(1,1,c)}(1)=1$ for $c<\sqrt2/2,$
$m(1,1,c)=f_{(1,1,c)}(\pi/4)=\sqrt{2}-c/2$ for $c>1,$ and
$m(1,1,c)=\min\{1,\sqrt{2}-c/2\}$ for $c\in[\sqrt2/2,1].$
Therefore (i) follows from Lemma \ref{l2}.

\s

(ii) To prove this statement we will use the function
$g_{(1,1,c)}$ defined in (\ref{2}). Its derivative is given by
$$g'_{(1,1,c)}(t)= \frac{(\sin t-\cos t)(1+\sin t\cos t-c(\sin t+\cos
t))}{\sin^2t\cos^2t}.$$ Note that the equation
$$1+\sin t\cos t-c(\sin t+\cos t)=0$$ has two distinct roots in $\DD^0$
(with sum $\pi/2$) for $c\in(1,\frac{3\sqrt{2}}{4})$ and a double
root $\pi/4$ for $c=\frac{3\sqrt{2}}{4};$ otherwise, it has no
roots in $\DD^0.$ Hence
$\min_{\Delta^0}g_{(1,1,c)}=g_{(1,1,c)}(\pi/4)=2\sqrt2-2c$ for
$c\le 1,$
$\min_{[t_1,t_2]}g_{(1,1,c)}=g_{(1,1,c)}(t_1)=g_{(1,1,c)}(t_2)=
c(1-\sqrt{c^2-1})$ for $c>\sqrt6-\sqrt 2,$ and
$\min_{[t_1,t_2]}g_{(1,1,c)}=\min\{2\sqrt2-2c,c(1-\sqrt{c^2-1})\}$
for $c\in[1,\frac{3\sqrt{2}}{4}].$ Note that for $c\ge 1,$
$$2\sqrt2-2c\le c(1-\sqrt{c^2-1})\Leftrightarrow c^2(c^2-1)\le(3c-2\sqrt2)^2$$
$$\Leftrightarrow(c-\sqrt2)^2(c-\sqrt6+\sqrt2)(c+\sqrt6+\sqrt2)\le 0.$$
Therefore (ii) follows from the proof of Lemma \ref{l3} and Remark
\ref{r1}.\qed

\begin{rem}\label{r2} Let $a=b=1$ and $R\ss\A_1$ be a rectangle with side
lengths $c$ and  $d.$ As we know from the proof of Proposition
\ref{p2} (and also from the above corollary), if $c\geq 1$ and $R$
can be moved into the corridor $\B_3$ by an anti-rotation, then
the same can be done by a rotation. Note, however, that for $c<1$
this is not always true. For example, if $d=1$ and $c\in
(2(\sqrt{2}-1), \sqrt{2}-1/2]$ it follows from Corollary \ref{cr
4} that $R$ can be moved into the corridor $\B_3$ by an
anti-rotation, but this cannot be done by a rotation.
\end{rem}

\section{Proof of Theorem \ref{t1}}

The equalities for the sets $\R_{ij}$ follow from their
characterizations in statements (i)--(iv) which we prove below.

\subsection{Proof of the \emph{if part} of Theorem \ref{t1}}

We will prove the \emph{if part } of each of the statements
(i)-(iv).

(i) It follows  from the proof of the \emph{if part } of
Proposition \ref{p1} since the movement for the rectangles in
$\R_{00}$ used there works for the rectangles in $\R_{01}$ and
$\R_{02}$ as well.

\s

(ii) Let $h \le d\le a\w b$ and $cd\le ab $. Then we move the
rectangle $R=KLMN $ as in the proof of Proposition \ref{p1} to a
position where $K=A,L\in \B_3$ and $ M,N\in\A_0$. We can assume
that $N\in Oy^+$ since otherwise we replace the ``wall'' $Oy$ of
the corridor $\A_0$ with the line $\{x=a-a'\},$ where $a'$ is the
distance from $N$ to the line $\{x=a\}$. The new corridor has
width $a'<a$ and the given conditions are satisfied since $d\ge
h>h'$ and we know from the proof of Proposition \ref{p1} that
$cd\le a'b$. Consider the tangent to $k(C;d)$ from $A(a,0)$ and
use the same notations as in the proof of the \emph{only if part}
of (ii) below (Fig.~\ref{fig3B}). Then
$$d=\frac{ab}{\sqrt{a^2+a_2^2}} \ , \
h=\frac{ab}{\sqrt{a^2+b^2}}$$ which imply  that $d>h
\Leftrightarrow a_2<b .$ We have also that  $c\le ab/d = |AA_2|$.
Using Remark $\textbf{(B)}$ in Appendix we conclude that this
inequality holds for the tangents to $k(C;d)$ from the points on
$Ox$ with abscissa $\ge a$. Now it follows from the proof of Lemma
\ref{l2} that one can move $R$ in $\B_3$ by a rotation around the
corner $C$.

Finally, if $d\le h\w m(a,b,c) $, then the desired movement of $R$
in $B_3$ is the same as that in the proof of the \emph{if part} of
Lemma \ref{l3} .

\s

 \begin{figure}[!t]
 \begin{center}
     \begin{subfigure}[]{0.38\textwidth}
     \includegraphics[width=\textwidth]{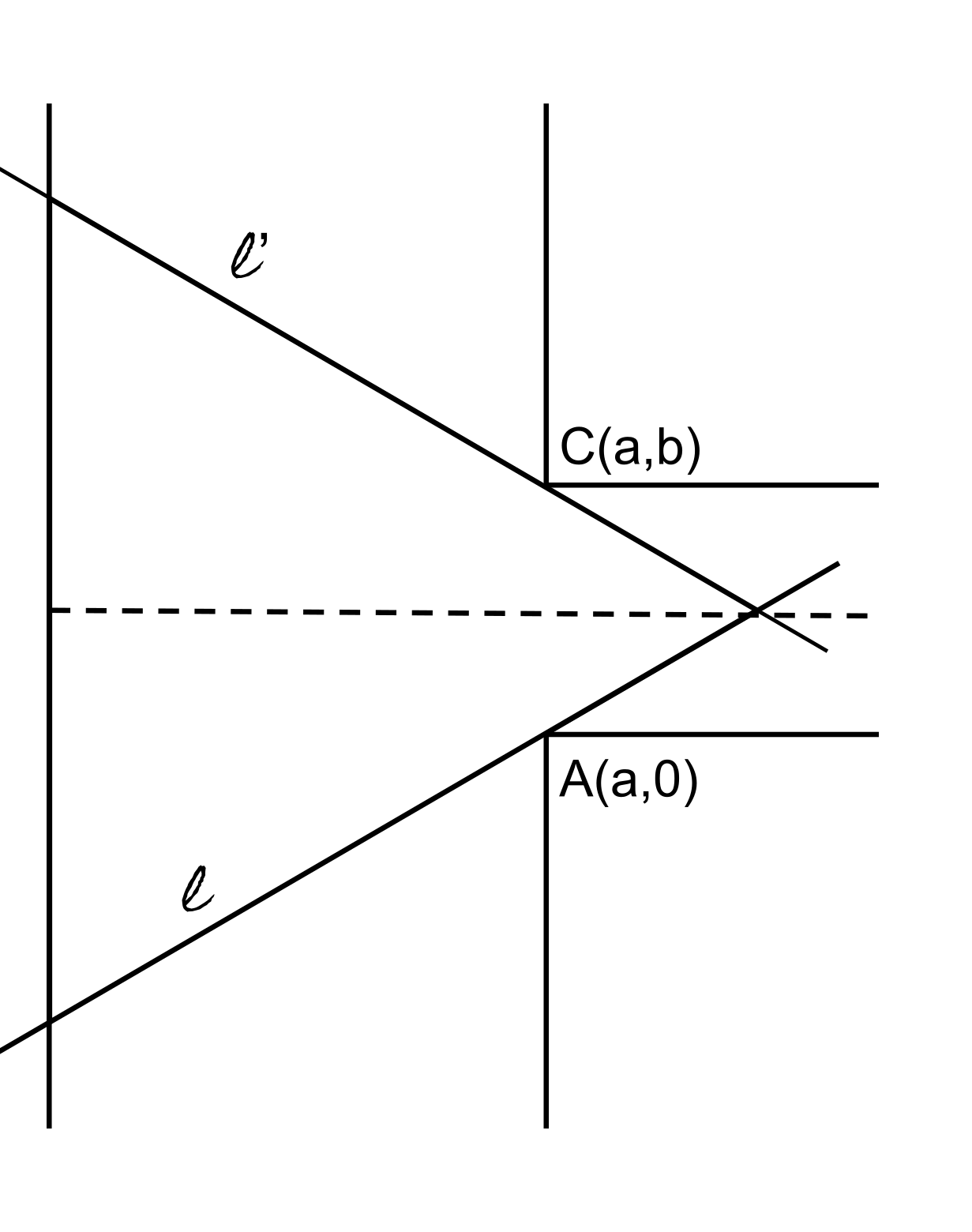}
     \caption{}
     \label{fig3A}
     \end{subfigure}
     \hspace{1cm}
     \begin{subfigure}[]{0.42\textwidth}
     \includegraphics[width=\textwidth]{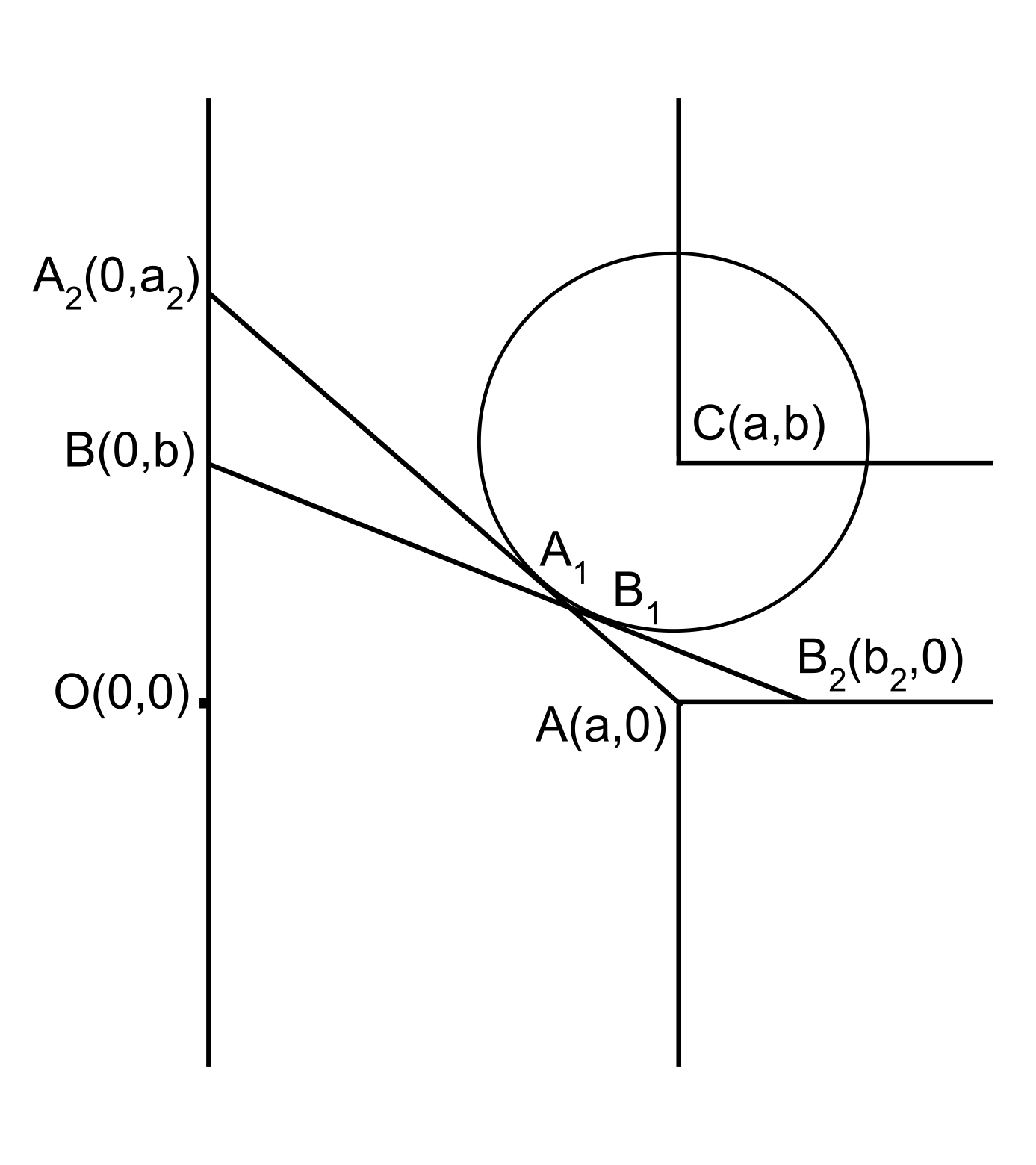}
     \caption{}
     \label{fig3B}
     \end{subfigure}
 \caption{The corridor $\C_{03}$.}
 \label{fig3}
 \end{center}
 \end{figure}

\s

(iii) Since $\R_{12}\ss \R_{10}$ we have to prove that a rectangle
$R$ in $\A_1$ with side lengths satisfying one of the three
conditions can move in $\B_2 .$ If $d\le h,$ then it follows from
Remark \textbf{(B)} in Appendix that $[AA_2]$ is the shortest
segment with ends on $Oy$ and $[OA],$ respectively which is
tangent to the circle $k(C;d)$. This segment has length $ab/d$ and
it follows from Proposition \ref{p1} and Proposition \ref{p2} that
in all three cases the rectangle $R$ can be moved from $\A_1$ into
$\B_2.$ Note that to apply Proposition \ref{p1} we need to show
that the condition $d\leq m(a,b,c)$ implies the inequality $cd\leq
ab .$ To see this set $\varphi= \arccos (a/l). $ Then

\begin{equation}\label{8}
cd\leq cf_{(a,b,c)}(\varphi)=abc(2/l-c/l^2)\leq ab
\end{equation}
with equality only if $c=l, d=h$.

\s

(iv) It follows from the proof of the \emph{if part } of
Proposition \ref{p2} since the movement for the rectangles in
$\R_{13}$ used there is not affected by the upper wall of the
corridor  $\B_3$ and it works for the rectangles in $\R_{11}$ as
well.

\n\subsection{Proof of the \emph{only if part} of Theorem
\ref{t1}}

We will prove the \emph{only if part } of each of the statements
(i)-(iv).

\s

(i) It follows from Proposition \ref {p1} since $\R_{01}\cup
\R_{02}\ss \R_{00}$.

\s

(ii) Let $R\ss \A_0$ be a rectangle with side lengths  $c$ and $d,$
$c\ge d,$ which moves in $\B_3,$ around the corner $C$. If $d>h$
the \emph{only if part} follows from Proposition \ref{p2}. Let now
$d<h$. We can assume that the initial position of $R$ is such that
its sides with length $c$ are parallel to $Oy$, while at its final
position they are parallel to $Ox$. Let $l$ be the tangent to the
circle $k(C,d)$ for which the minimum $n(a,b,d)$ of $g_{(a,b,d)}$
is attained and $l'$ be the line which is symmetric to $l$ with
respect to the line $y=b/2$ (Fig.~\ref{fig3A}). Then at some
moment the sides of $R$ with length $c$ are parallel either to
$l$, or to $l'$ and by symmetry, we can assume that these sides
are parallel $l$. Consider the tangents from $A$ and $B$ to the
circle $k(C,d)$ such that they both intersect $Ox^+$ and $Oy^+$.
Denote by $A_1$ and $B_1$ their tangent points to $k(C,d)$, and by
$A_2=(0,a_2)$ and $B_2=(b_2,0)$ their intersection points with
$Oy^+$ and $Ox^+$, respectively (Fig.~\ref{fig3B}). Since $d<h$ it
follows that $a_2> b $ and $ b_2> a $ which shows that the line
$l$ intersects the segments $[B,A_2]$ and $[A,B_2]$. Now it is
easy to see that at the moment when the sides of $R$ with length
$c$ are parallel $l$ we have $c\le n(a,b,d)$ whatever is the
position of $R$. Hence $d\le m(a,b,c) $ and the \emph{only if
part} for $d<h$ is also proved.

\s

(iii) It follows from Proposition \ref{p1} that it is enough to
prove that if  $h<d<a\w b,$ $c>a\vee b$ and $R\in\R_{10},$ then
$d\le m(a,b,c).$

Let $\ell$ be the line through $O$ which is parallel to the sides
of $R$ with length $c$ and let  $\v\in [0,\pi/2]$ be the angle
between $\ell$ and $Ox$. As in the proof of Proposition \ref{p1}
we can assume that $\v=\pi/2$ in the initial position of $R$.
Consider the last moment until at least three of its vertices lie
in $\A_{1}$. Having in mind that $c>a\vee b,$ it follows that at
that moment $R=KLMN,$ where $|KL|=d$ and $K\in\D,$
$L\in\B_1\setminus\D,$ $M,N\in\A_1\setminus\D .$ Let $y_K$ and
$y_L$ be the ordinates of $K$ and $L$. It is clear that $y_K\neq
y_L$. If $y_K>y_L,$ then  $R\subset \B_3$ and the inequality
$d\le m(a,b,c)$ follows from Proposition \ref{p2}. Let now
$y_K<y_L$. Then one can easily see that $\v\le\psi,$ where $\psi$
is the angle corresponding to the tangent to the circle $k(C;d)$
through $A$. Since $d>h,$ it follows from Remark $\textbf{(B)}$ in
Appendix that $\psi<\psi_0,$ where $\psi_0$ is the angle
corresponding to the tangent to $k(C;d)$ for which the minimum
$n(a,b,d)$ of the function $g_{(a,b,d)}$ is attained. It follows
by continuity that at some moment the line $\ell$ is parallel to
that tangent. For such a position of $R$ it is clear that $c\le
n(a,b,d)$ which is equivalent to $d\le m(a,b,c).$

\s

(iv) It follows from Proposition \ref{p2} since $\R_{13}\ss
\R_{11}$.\qed

\section{Proof of Theorem \ref{t2}}

Let  $P$ be a rectangular parallelepiped that can move around the
corner of the corridor $S_{ij}$. We will show that $V_P\le abc .$
To do  this consider the last moment until at least six vertices
of $P$ lie in $\A_i\times (0,c)$. Set $\D'=\o{\D}\times[0,c].$
Then at least three vertices $K, L , M$ of a face of $P$ lie in
$\B_j\times (0,c)\cup\D'$ with $M\in\D' $ and the corresponding
vertices $K_1, L_1, M_1$ of its opposite face lie in
$\o{\A_{i}}\times[0,c].$ It is clear that the edges $KK_1$ and
$LL_1$ meet $\D' $ at some points $K_2$ and $L_2.$ Let the line
$MM_1$ meet the plane $\{x=0\}$ at a point $M'$ and let $\theta$
be the angle between this plane and the plane $(KLM).$ Then
$$2S_{\t KLM}=2S_{\t K_2L_2M}\cos\theta\le bc\cos\theta,\ MM_1\le
MM'= \frac{a}{\cos\theta},$$ and we get $V_{P}\le abc.$

If $V_{P}=abc,$ then $2S_{\t K_2L_2M}=bc,$ i.e. two of the
vertices of $\t K_2L_2M$ are adjacent vertices of $\D',$ and the
third one is lying on its opposite side. Then it is easy to see
that $P=R\times(0,c),$ where $R\in\R_{ij}.$ (In fact $M_1=M'_1$
and $\D'=P\cap \{x=0\} .$ ) Now Theorem \ref{t2} follows by Lemma
\ref{l1}  which we prove below.

\s

\n{\it Proof of Lemma \ref{l1}}. Let $\C\ss P$ be a right
circular cylinder whose basses circles lie on $\alpha$ and $\beta$
and have diameter $r>0$. Suppose that at some moment of the
movement of $P$ in the layer between $\alpha$ and $\beta$ it, and
hence $\C$, is inclined  to these planes. Consider the
intersection of  $\C$ with the plane through the centers of its
basses circles and orthogonal to $\alpha$ and $\beta$. It is a
rectangle $R=KLMN$ with $|KL|=|MN|=c$(the distance between
$\alpha$ and $\beta$) and $|NK|=|ML|=r$. Let $NK\cap \alpha=K_0,
MN\cap \alpha=N_0$ and $\angle N_0K_0N=\varphi_0$. We can assume
that $\C$ is inclined to $\alpha$ and $\beta$ such that
$\varphi_0\in (0 , \pi/2)$.Then

\begin{equation}\label{4}
r\le |N_0N|\cot\varphi_0 \le\left(\frac
{c}{\cos\varphi_0}-c\right)\cot\varphi_0=c\tan\frac{\varphi_0}{2} .
\end{equation}
and it follows by continuity that this inequality holds for all
$\varphi \in (0 , \varphi_0)$. Now letting $\varphi\rightarrow 0+$
in (\ref{4}) we get $r\le 0$, a contradiction. \qed

\section{Appendix}

In this section we prove  some technical facts that are used in
the proofs of the results stated in Section 1.

\s

\n{\bf A.} Note that in some particular cases the minimum
$m(a,b,c)$ of the function $f_{(a,b,c)}$ can be found explicitly.
For example, it follows by Remark \textbf{(B)} below that
$m(a,b,l)=h$ and $m(a,b,c)=0$ for $c\ge (a^{2/3}+b^{2/3})^{3/2}
.$  We also know from Corollary \ref{cr 4} that
$m(a,a,c)=a\sqrt{2}-c/2$ for $c\ge a$.

In the general case $m(a,b,c)$ is a rational function of the
largest root of a quartic equation and hence can be found in
radicals with respect to $a,b,c .$ Indeed, the derivative of
$f_{(a,b,c)}$ is given by
$$f_{(a,b,c)}'(t)=a\cos t-b\sin t-c\cos 2t$$ and setting $x=\tan
t$ for $|t|<\pi/2$  we get
$$p(x):=(x^2+1)f'_{(a,b,c)}(\arctan x)=(a-bx)\sqrt{x^2+1}+c(x^2-1).$$
We set also $$q(x)=(a-bx)\sqrt{x^2+1}+c(1-x^2)$$ and note that
$r=pq$ is a polynomial of degree $4$. One checks easily that
$p(\pm\infty)=+\infty,$ $p(0)<0,$ $q(\pm\infty)=-\infty$ and
$q(0)>0.$ Therefore, each of the functions $p$ and $q$ has exactly
one positive and one negative zero and these are all four zeros of
$r.$ Denote by $x_1$ and $y_1$ the positive zeros of $p$ and $q$,
respectively. If $a=b$, then $x_1=y_1=1.$ If, for example $a<b ,$
then $p(1)=q(1)<0$. Hence $y_1<1<x_1$ and $x_1$ is the largest
zero of $r.$ Therefore we can find $x_1$ in radicals with respect
to $a,b,c ,$ and do the same for $m(a,b,c)=f(\arctan x_1).$

\s

\n{\bf B.} For fixed $a, b> 0$ and $0\le d\le a\w b $ consider
$n(a,b,d):=\inf_{\DD^o}g_{(a,b,d)},$ where
$$g_{(a,b,d)}(t):=\frac{a\sin t+b\cos t-d}{\sin t\cos t},\ t\in\DD^o.$$
Recall that $g_{(a,b,d)}(t)$ is the length of the segment which
the tangent to the circle $k(C;d)$ at the point $(a-d\sin
t,b-d\cos t)$ cuts from the first quadrant. Hence $n(a,b,d)$ is
the length of the shortest such segment.

Straightforward computations show that
$$g'_{(a,b,d)}(t)=\frac{a\sin^3t-b\cos^3t+d\cos2t}{\sin^2t\cos^2t}$$
and
$$g''_{(a,b,d)}(t)\sin^3t\cos^3t=b(\cos^3t+\cos^5t)+a(\sin^3t+\sin^5t)-d(1+\cos^22t)\ge$$
$$d((1-\sin t)^2\sin^3t+(1-\cos t)^2\cos^3t)\ge 0 .$$ Hence $g_{(a,b,d)}$ is a
convex function. Since  $g'(0+)=-\infty$ for $d<b,$ $g'(0+)=-d/2$ for $d=b,$
$g'(\pi/2-)=+\infty$ for $d<a$ and $g'(\pi/2-)=d/2$ for $d=a$ it
follows that $g'$ has a unique zero $t_0\in\DD^o$ and
$n(a,b,d)=g_{(a,b,d)}(t_0).$

Set $x=\tan t$. Then the equation $g_{(a,b,d)}'(t)=0$ takes the
form  $u(x)=0,$ $x>0,$ where
$$u(x)=ax^3-b-d(x^2-1)\sqrt{x^2+1}.$$
We know that $u$ has a unique positive zero $x_0=\tan t_0$ and
this zero is simple since $g_{(a,b,d)}''>0.$ As in \textbf{(A)} we
set
$$v(x)=ax^3-b+d(x^2-1)\sqrt{x^2+1}$$
and note that $uv$ is a polynomial of degree $6$. One easily
checks that  $v'(x)>0$ for  $x>0,$ $v(0)<0$ and
$v(+\infty)=+\infty$ which show that $v$ has a unique positive
zero $x_1$ and this zero is simple. Hence $x_0$ and  $x_1$ are the
only positive zeros of $uv.$ If $a=b ,$ then $x_0=x_1=1,$ and if,
for example, $b<a,$ then $0<x_0<\sqrt[3]{b/a}<x_1<1$ (compare with
\cite{Mo}).

It is clear that $n(a,b,d)$ is a continuous and strictly
decreasing function in $d\in[0,a\w b].$ Let $d\neq 0$ and
$\varphi= \arccos (a/l). $ Then $$g_{(a,b,d)}(\varphi)=
ab\left(\frac{2}{l}-\frac{d}{l^2}\right)\le\frac{ab}{d} $$ which implies
that
\begin{equation}\label{7}
n(a,b,d)\le \frac{ab}{d}
\end{equation}
with equality if and only if $d=h$. Hence $n(a,b,h)=l$ and
$m(a,b,l)=h .$

Now consider the case $d=0$ which corresponds to the \emph{leader
problem} \cite{Ka} .It follows easily by the above results that in
this case $\tan t_0=(b/a)^3$ and
$n(a,b,0)=(a^{2/3}+b^{2/3})^{3/2}$. This together with Theorem
\ref{t1} implies Corollary \ref{cr1}.

Finally, if $a=b ,$ then $t_0=\pi/4$ and $n(a,a,d)=2\sqrt2a-2d.$
which we know from Corollary \ref{cr 4}.

Note that the above three particular cases are related to
\cite[Questions 2 and 6, p. 200]{Mo}.

\s

\n{\bf C.} The condition  $d\le m(a,b,c)$ for rotation around the
corner $C$ (see e.g. Lemma \ref{l2}) can be written as $a\ge
k(c,d,b):=\sup_{\DD'}h_{(c,d,b)},$ where $\DD'=(0,\pi/2]$ and
$$h_{(c,d,b)}(t)=c\cos t+\frac{d}{\sin t}-b\cot t.$$
It is clear that $d\le k(c,b,d),$ and that $k(c,d,b)<+\infty$ is
equivalent to $d\le b.$

We have
$$w(t):=h'_{(c,d,b)}(t)\sin^2t=b-c\sin^3t-d\cos t,\ 2w'(t)=\sin t(2d-3c\sin2t).$$
If $2d<3c,$ then $w$ has two zeros  $t_1<t_2\in \DD^0 ,
t_1+t_2=\pi/2$, $w$ is strictly increasing in the intervals
$[0,t_1]$ and $[t_2,\pi/2] ,$  and it is strictly decreasing in
the interval $[t_1,t_2].$ In particular (c.f. \cite{Mi}) if  $b\le
c ,$ then $w$ has a unique zero $t_\ast\in\DD^0$  and $t_1<t_1'\le
t_\ast\le t_2'<t_2,$ where $w(t_1')=w(0)$ and $w(t_2')=w(\pi/2)$.
Therefore $k(c,d,b)=h_{(c,d,b)}(t_\ast)\ge d.$ (Note that
similarly to \textbf{(B)}, $\arctan t_\ast$ is a zero of a six
degree polynomial.)

Having in mind the inequality $k(c,d,b)\ge d$ it is natural to ask
when the equality is attained. In other words the question is when
given $b\ge d$ and $c ,$ the least $a$ such that $m(a,b,c)\ge d,$
is equal $d.$ We will show that this is so if $c\le(\sqrt
2-1/2)d.$ To prove this it is enough to check that $m(a,b,c)=d$ if
$a=b=d$ and $c=(\sqrt 2-1/2)d $ which follows from \textbf{(A)}
(as well as from \textbf{(B)}, interchanging $c$ and $d$). Note
that the above constraint for c is the optimal one.

One can prove in the same way that given $a\le b$ and $c\le(\sqrt
2-1/2)a,$ the largest $d$ for which $m(a,b,c)\ge d$ is equal to $a.$

\end{document}